\documentclass{amsart}

\newtheorem{lemma}{Lemma}[section]
\newtheorem{thm}[lemma]{Theorem}
\newtheorem{prop}[lemma]{Proposition}

\newcommand{\arr}[2]{\begin{array}{#1}#2\end{array}}

\newcommand{\matr}[2]{\left(\begin{array}{#1}#2\end{array}\right)}

\usepackage{graphicx}
\begin{document}

\title{On  McKay Quiver and Covering Spaces}
\author[Guo]{Jin Yun Guo\\
Department of Mathematics\\ Xiangtan University\\ Xiangtan,  CHINA}
\thanks{This work is partly supported by
Natural Science Foundation of China \#10671061, SRFDP\#200505042004}
\email{ gjy@xtu.edu.cn} \dedicatory{}

\subjclass{Primary {16G20, 20C05}; Secondary{ 16S34,16P90}}

\keywords{McKay quiver; covering space; Nakayama translation; finite group}

\date{}

\begin{abstract}
In this paper,  we study the relationship between the McKay quivers of a finite
subgroups $G$ of special linear groups general linear groups, via some natural
extension and embedding.  We show that the McKay quiver of certain extension of
a finite subgroup $G$ of $\mathrm{SL}(m,\mathbb C)$ in  $\mathrm{GL}(m,\mathbb C)$  is a
regular covering of the McKay quiver of $G$, and when embedding $G$ in a
canonical way into  $\mathrm{GL}(m-1,\mathbb C)$, the new McKay quiver is obtained by
adding an arrow from the Nakayama translation of $i$ back to $i$ for each $i$.
We also show that certain interesting examples of McKay quivers are obtained in
these two ways.
\end{abstract}

\maketitle


In 1980,  John McKay introduced McKay quiver for a finite subgroup of the
general linear group \cite{mc}. McKay observes that when $G$ is a subgroup of
$\mathrm{SL}(2, \mathbb C)$,  then its McKay quiver $Q = Q_G $ is a double quiver of
the affine Dynkin diagram of type $A, D, E$ \cite{mc}. McKay observed that the
McKay quivers describe the relationship between these groups and the Kleinian
singularities.

McKay quiver has bridged many mathematical fields such as algebraic geometry,
mathematics physics and representation theory (See,  for example, \cite{reid}).
In representation theory of algebra, for example, it appears in the study of
the Auslander-Reiten quiver of Cohen-Macaulay modules \cite{a, ar},
preprojective algebras of tame hereditary algebras \cite{cb, cb1, dr, bgl} and
quiver varieties \cite{na},  etc. We find that it can also play a critical role
in classification of selfinjective Koszul algebras of complexity 2 \cite{gyz}.

Covering space is an important tool of the algebraic topology, it is
introduced in the study of the representation theory of algebra  by
Bongartz and Gabriel in \cite{gr} and plays an important role
here\cite{ds, px, kr}. In this paper, we show that the covering maps
also appear naturally in the McKay quivers. We show that when one
extend a finite subgroup of $\mathrm{SL}(m,\mathbb C)$ via some nice cyclic
group to a finite subgroup of $\mathrm{GL}(m,\mathbb C)$, one gets natural
covering for their McKay quivers. There are also covering maps
between such extensions. So one can even construct the 'universal
cover' of an McKay quiver.

Such coverings are not only geometric for the McKay quivers. They
also induce  coverings in the categories of the projective-injective
modules of the corresponding skew group algebras over the  exterior
algebra of the given vector spaces. So their representation theory
are related. Using the universal covering, one get an "universal
algebra" which has simply connected quiver. This can be used to
explain our early results for the case of $m=2$ \cite{gmt}.

Though McKay quivers are well known for $\mathrm{SL}(2, \mathbb C)$.  It is difficult
to determine the McKay quiver in general.  We describe the McKay quiver for
certain  finite subgroup of general linear group. Let $G$ be a finite subgroup
of the general linear group, then we have that $N= G\cap \mathrm{SL}(m, \mathbb C)$ is
a normal subgroup of $G$ and $\bar{G} = G/N$ is a finite cyclic group. We
observe that under certain condition on $G$,  the McKay quiver of $G$ is a
covering quiver of the McKay quiver of $N$ with the group $\bar{G} $. We also
show by example, how the "physics" quivers appearing in the study of D-branes
\cite{gj} can be explained using our results.

\section{Finite Subgroups in $\mathrm{GL}(m, \mathbb
C) $ and $\mathrm{SL}(m, \mathbb C) $}

Let $V$ be an $m$-dimensional vector space over $\mathbb C$ and let
$G$ be a finite subgroup of $\mathrm{GL}(m, \mathbb C)= \mathrm{GL}(V)$. Let $Q_G =
(Q_{G,0}, Q_{G,1})$ be the McKay of $G$.  $V$ is naturally a
faithful representation of $G$. Let $\{S_i | i =1,  2, \ldots, n\}$
be a complete set of irreducible representations of $G$ over
${\mathbb C}$.  For each $S_i$, decompose the tensor product $V
\otimes S_i$ as a direct sum of irreducible  representations, write
$$V \otimes S_i = \bigoplus_j a_{i, j} S_j,  \quad\mbox{  } \quad i = 1,  \ldots,  n, $$
here $a_{i, j} S_j$ denotes a direct sum of $a_{i, j}$ copies of
$S_j$. $a_{i, j}$ is finite since $V$ is finite dimensional.  The
McKay quiver $Q=Q_G$ of $G$ is defined as follow.  The vertex set
$Q_{G,0} $ is the set of  the isomorphism classes of the irreducible
representations of $G$,  and there are $a_{i, j}$ arrows from the
vertex $i$ to the vertex $j$.


The need the following lemma.

 \begin{lemma}\label{cyc}
 Let $G$ be a finite subgroup of $\mathrm{GL}(m, \mathbb
C)$ and let $ N= G\cap \mathrm{SL}(m, \mathbb C)$.  Then $G/N$ is a cyclic
group.  \end{lemma}

\begin{proof} Consider the map $\mathrm{det\,}: G \to \mathbb C^*$,  which send each
matrix $g \in G$ to its determinant.  Clearly $\mathrm{det\,}$ is a
homomorphism with the kernel $N$.  So $G/N$ is a subgroup of
$\mathbb C^*$,  and hence is a finite abelian group.  By the
Fundament Theorem of the Structure of Finite Abelian Groups,  $G/N$
is a direct sum of cyclic groups.  Let $\bar{g},  \bar{h}$ be two
elements in $G/N$ such
that $|\bar{g}| | |\bar{h}|$.  Then there is an $ |\bar{h}|$th root $\xi$ of
the unit,  such that $\mathrm{det\,}(h) = \xi$ and $\mathrm{det\,}(g) = \xi^{\frac{
|\bar{h}|}{|\bar{g}|}}$.  So we have that $\mathrm{det\,}(g h^{-\frac{
|\bar{h}|}{|\bar{g}|}}) = 1$ and $g h^{-\frac{ |\bar{h}|}{|\bar{g}|}} \in N$.
This implies that $\bar{g} = \bar{h}^{\frac{ |\bar{h}|}{|\bar{g}|}}$ and one
see easily that $G/N$ is cyclic.
 \end{proof}


Let $Q$ be a quiver and $i \in Q_0$ be a vertex. Denote by $i^-$ the
set of arrows ending at $i$ and $i^+$ the set of arrows starting at
$i$.

Let $G$ be a group, a map $\pi: Q \to Q'$ of quivers is called a
regular covering map with the group $G$ provided that for any vertex
$j\in Q'_0$ the group $G$ acts transitively and freely on
$\pi^{-1}(j)$; and for each vertex $i \in Q_0$ the induced maps
$\pi_{i^+}: i^+ \to \pi(i)^+$ and $\sigma_{i^-}: i^- \to
\sigma(i)^-$ are both  bijective. Clearly, this is exactly the
regular covering map of oriented linear graphs with automorphism
group $G$. (see Chapter 5 and 6 of \cite{msy}). When $Q=Q'$ and
$G=\{1\}$, a covering map is just an automorphism of the quiver,
that is, an automorphism is a regular covering map with the trivial
group.

 \begin{thm}\label{cover} Let $G$ be a finite subgroup of $\mathrm{GL}(m, \mathbb
C)$ and let $ N= G\cap \mathrm{SL}(m, \mathbb C)$.  If every irreducible
character of $N$ is extendible, 
then the McKay quiver of $G$ is a regular covering of the McKay quiver of $N$
with the automorphism group $G/N$.  \end{thm}
 \begin{proof}
Since every irreducible character of $N$ is extendible, If
$\chi_1,\ldots, \chi_n$ are the irreducible characters of $N$, there
exist irreducible characters $\chi'_1,\ldots, \chi'_n$  of $G$, such
that $\chi'_i|_N =\chi_i $, for $i =1,\ldots, n$.  Assume  the
characters $\chi_1',\ldots, \chi_n'$ are afforded by the irreducible
representations $S_1,\ldots, S_n$ of $G$. Since the irreducible
characters of $N$ are extendable,   $S_1,\ldots, S_n$ are also
irreducible representations of $N$, And they afford the irreducible
characters $\chi_1,\ldots, \chi_n$ of $N$, respectively. If
$\beta_1,\ldots , \beta_r$ are the irreducible characters of $G/N$
which are afforded by the irreducible representations $T_1,\ldots ,
T_r$ of $G/N$, they are naturally regarded as the representations of
$G$. Since $G/N$ is cyclic by Lemma~\ref{cyc}, $|G/N| = r$ and
$T_1,\ldots , T_r$ are all $1$-dimensional. By Corollary 6.17 of
\cite{is}, $\{\chi_i'\beta_j| 1\le i \le n, 1\le j \le r \}$ is a
set of irreducible characters of $G$ and $\chi_i'\beta_j =
\chi_s'\beta_t$ if and only if $i=s$ and $j=t$. Clearly,
$\chi_i'\beta_j $ is afforded by the representations $S_i \otimes
T_j$ of $G$, so $S_i \otimes T_j$ is  irreducible for $1\le i \le n,
1 \le j \le r$. Since $|G| = |N| \cdot |G/N| $, by comparing the
dimensions, we find that $\{S_i \otimes T_j | 1\le i \le n, 1 \le j
\le r\}$ is a complete set of the irreducible representations of
$G$.

If  $V = \bigoplus_{i=1}^{n} b_i S_i$ over $N$, then there is a
$j_1,\ldots,j_n \in \{1,\ldots,n\}$ such that $V=
\bigoplus_{i=1}^{n}\bigoplus_{j=1}^{r}  b_{i,j} S_i \otimes T_{j}$. Then $b_i= \sum_{j=1}^{r}  b_{i,j}$ 
Let 
$\bar{g}$ be a generator of $G/N$, there is a $r$th root $\xi$ of the unit such that $\bar{g}
x = \xi^{l_j} x $ for any $x \in T_j$. Reindex the irreducible
representations if necessary, we may assume that the index are taken
from the set of residue classes modulo $r$. Then $ T_i \otimes
T_{j_i} \simeq T_{i + j_i}$.
So we have that for each $i, k \in\{ 1, \ldots, n \}$, if $S_i
\otimes S_k \simeq \bigoplus_{k=1}^{n} b_{i,k}c_{i,k;j} S_j$, and $
S_i \otimes V \simeq \bigoplus_{j=1}^{n} a_{i,j} S_j$ as
representation of $N$, then $a_{i,j} = \sum_{k=1}^n b_{i,k}
c_{i,k;j}$. As representation of $G$, we have that for $i = 1,
\ldots, n $ and $s = 1, \ldots r$,
$$(S_i \otimes T_s) \otimes   V \simeq
 \bigoplus_{k=1}^{n}S_i \otimes S_k \otimes T_s
 \otimes T_{j_k}\simeq \bigoplus_{k=1}^{n} \bigoplus_{l=1}^{n}  c_{i,k;l}
 S_l
\otimes T_{s+j_k}.$$

So we get that if $Q_{N,0} = \{i|i=1, 2, \ldots, n\} $, then $Q_{G,0} =
\{(i,t)| i = 1,2,\ldots,n; t  = 1,\ldots,r\}$. And we see that if there are
$a_{i,l}$ arrows from $i$ to $l$ in $Q_N$, then there are $b_{i,k}c_{i,k;l}$
arrows from $(i,s) \to (l,s+j_k)$ in $Q_G$, and as $k$ ranges over $1,\ldots,
n$, we get all together $a_{i,j}$ arrows from $(i,s)$ to $(j,s+j_1), \ldots
,(j,s+j_n)$. So we see that the restricting map $S_i\otimes T_s \to S_i$
induces a map $\pi:Q_{G,0} \to Q_{N,0}$ such that $\pi(i,t) = i$, whose fibre
consists of single orbit of $G/N$ with $G/N$ acts freely. It also induces a
bijection $(i,s)^+$ and $i^+$ in the McKay quiver and so it induces a regular
covering map from $Q_G$ to $Q_N$ with the group $G/N$.
 \end{proof}

{\bf Remark.} \begin{itemize}
\item Clearly, for any finite subgroup $N \subset \mathrm{SL}(m,\mathbb C)
$ and any positive integer $r$,  let $\xi_{mr}$ be an $mr$th root of
the unit. Let $H_r$ be the subgroup generated by the diagonal matrix
$\left(\arr{cccc}{\xi_{mr}&0 &\cdots& 0\\0& \xi_{mr}&\cdots& 0\\
\cdot &\cdot &\cdots& \cdot\\0&  0 &\cdots&\xi_{mr}\\} \right)$.
$H_r$ is a cyclic subgroup of $\mathrm{GL}(m, \mathbb C)$ of order $r$. The
product $G = N \times H_r$ satisfy the conditions of Theorem ~\ref{cover}.
So the McKay quiver of $G$ is a regular covering space of the McKay
quiver of $N$ with the group $H_r$. So given a finite subgroup $G$
of $\mathrm{SL}(m, \mathbb C)$ and finite cyclic group $H$, we can always
find a subgroup of $\mathrm{GL}(m, \mathbb C)$ whose McKay quiver is a
regular covering of that of $G$ with the group $H$.

\item In the case $m=2$ it follows from \cite{gmt, g2} that
the McKay quiver of a finite subgroup of $\mathrm{GL}(2,\mathbb C)$ should
be a regular covering graph of the McKay quiver of some finite
subgroup of $\mathrm{SL}(2,\mathbb C)$(a double quiver of some affine
Dynkin diagram). They are in fact a translation quiver with a slice
of affine Dynkin quiver. But we don't know the exact relationship
between them.

\end{itemize}

In fact, given a finite subgroup   $G \subset \mathrm{GL}(m,\mathbb C) $
such that  every irreducible character of $N = G \cap \mathrm{SL}(m,\mathbb
C) $ is extendible, the covering maps exist not only for the McKay
quivers of $G$ and the subgroup $N$, but also  for the McKay quivers
of any subgroup  $N \subset L \subset G$. It is obvious that such
$L$ is a normal subgroup and  every irreducible character of $L$ is
extendible. In fact, our proof of Theorem ~\ref{cover} tell us the
condition in the following Theorem already guarantees the existence
of a covering map.

 \begin{thm}\label{coverg} Let $G$ be a finite subgroup of $\mathrm{GL}(m, \mathbb
C)$ and $L$ be a normal subgroup of $G$ such that $G/L$ is cyclic. If every
irreducible character of $L$ is extendible. Then the McKay quiver of $G$ is a
regular covering space of the McKay quiver of $L$ with the group $G/L$. \end{thm}

Given a finite subgroup $G_0$ of $\mathrm{GL}(m, \mathbb C)$.  Let $\mathcal
G_m(G_0)$ be the set of finite subgroups $G$ of $\mathrm{GL}(m, \mathbb C)$
containing $G_0$ as a normal subgroup such that $G/G_0$ is cyclic and all the 
irreducible character of $ G_0$ is extendible to $G$. Then if $G \in
\mathcal G_m(G_0)$, we denote the regular covering map  for their
McKay quivers as $\pi_{G,G_0}: Q_{G} \to Q_{G_0}$. The following
theorem is obvious.

\begin{thm}\label{cmpcover} Let $L$ be a finite subgroup  of $\mathrm{GL}(m, \mathbb
C)$. $N \in\mathcal G_m(L)$ and $G \in\mathcal G_m(N)$. Then  $G
\in\mathcal G_m(L)$ and $\pi_{N,L} \circ \pi_{G,N}= \pi_{G,L}$.
\end{thm}

Thus the regular covering map $\pi_{G,N}$ is in fact a homomorphism from the
covering map $\pi_{G,L}$ to the regular covering map for $\pi_{N,L}$. For any
finite subgroup $L \subset \mathrm{SL}(m, \mathbb C)$, the McKay quivers for the
groups in $\mathcal G_m(L)$ together with the regular covering maps between
them is a category.

\section{Skew Group Construction and Morphisms of Graphs with Relations}

In \cite{gum}, we relate a finite subgroup of $\mathrm{GL}(m, \mathbb C)=
\mathrm{GL}(V)$ with a finite dimensional selfinjective Koszul algebra, its
skew group algebra over the exterior algebra $\wedge V$ of $V$, which has
the McKay quiver as its quiver. With the exterior construction, it
seems that we have a finer description of the regular covering of
the McKay quivers finer, by using the morphism  of quiver with
relations introduced by Green \cite{gre}.

Let $V$ be an $m-$dimensional vector space over $\mathbb C$ and let
$G$ be a finite subgroup of $\mathrm{GL}(V)$. Let $\wedge V$ be the exterior
algebra of $V$, construct the skew group algebra $\wedge V*G$ over the
exterior algebra $\wedge V$ using the natural action of $G$ on $V$,  we
know that the quiver of $\wedge V*G$ is exactly the McKay quiver of $G$
and the map of determinant introduces Nakayama translation on the
vertices of $Q_G$ \cite{gum}.

Let $\Lambda$ be a finite dimensional selfinjective graded algebra over
an algebraically closed field $k$, and let $Q$ be its quiver. The
Nakayama translation is defined as the permutation $\sigma$ on the
vertex set $Q_0$ of $Q$ such that for any $i \in Q_0$, $\sigma i$ is
the vertex of the simple which is the socle of the projective cover
of the simple associated to $i$ \cite{g4}. $\sigma$ defines a
permutation on $Q_0$, called the Nakayama translation.

Let $\Lambda(G)$ be the basic algebra of $\wedge V*G$, we call it the basic
algebra of $G$. Then there is an idempotent element $e \in \wedge V*G$,
such that $\Lambda(G) = e \wedge V*G e$. In fact, $e = e_1 + \cdots + e_n$
for a set $\{e_1, \ldots, e_n\}$ of orthogonal primitive idempotents
such that $\{\wedge V*G e_i| i=1,\ldots,n\}$ is a complete set of
representatives of the isomorphism classes of the indecomposable
projective modules. $\Lambda(G)$ is a finite dimensional selfinjective
Koszul algebra with gradation $$\Lambda(G) = \Lambda_0 + \Lambda_1 + \cdots
+\Lambda_{m}.
$$
$J= \Lambda_1 + \cdots +\Lambda_{m} $ is its radical and $J^t/J^{t+1} \simeq
\Lambda_t$ as $\Lambda_0$-module.

By \cite{gum}, we have that $\Lambda(G) \simeq \mathbb C Q_G/I$ for some
admissible ideal $I = (\rho_G)$ of the path algebra $ \mathbb C Q_G$
which is induced by the relations defining $\wedge V$. In this case, we
have $1 =(e =) \sum_{i \in Q_{G,0}} e_i$ is a decomposition of $1$
of $\Lambda(G)$ as orthogonal primitive idempotent elements. By
definition, the number of arrows from  $i$ to $j$ is $\mathrm{dim}_{\mathbb
C} e_j V*G e_i$. If $\{v_1, \ldots, v_m\}$ is  a bases of $V$, they
generates $V*G$ as an $kG$-$kG$-bimodule. We have that $\Lambda_0 = e kG
e$, $\Lambda_1 = e V*G e$ and $ e_j V*G e_i = e_j \Lambda_1 e_i$. Choose a
basis $v_1, \ldots, v_m$ of $V$, they defines the arrow of $Q_G$, in
fact, we have an arrow of type $t$ from $i$ to $j$ provided that
$e_j v_t e_i  \neq 0$.

Let $G$ be a finite subgroup of $\mathrm{GL}(V)$ and let $T_kV$ be the
tensor algebra of $V$, the skew group algebra $T_k*G$ of $G$ over
$T_kV$ has as its quiver the McKay quiver $Q_G$ of $G$. We have the
following lemma.

\begin{lemma}\label{COMMARR} If $\alpha$ and $\beta$ are two arrows of different
types starting at the same vertex in the quiver $Q_{G}$ of $\Lambda(G)$,
then there are arrows $\alpha'$ of the same type as $\alpha$ and
$\beta'$ of the same type as $\beta$, such that $\alpha\beta'$ and
$\beta\alpha'$ will ending at the same vertex. \end{lemma}

\begin{proof} Let $T_kV$ be the tensor algebra of $V$, it is well known that
the exterior algebra $\wedge V \simeq T_kV/(\rho_0)$, where $(\rho_0)$ is
the ideal generated by $v_i\otimes v_i$ and $v_i\otimes
v_j+v_j\otimes v_i$ for all $i,j$. If $G$ is a finite subgroup of
$\mathrm{GL}(V)$. In the skew group algebra $T_kV
* G$, we have that, if $e_l,e_h$ are idempotents in $kG \subset
T_kV$ and if $e_h = \sum_{g\in G'} a_g g, a_g \in k$, then
$$e_h v_i\otimes v_j e_l = \sum_{g\in G} g^{-1}(v_i)\otimes  g^{-1}(v_i) a_g g e_l,$$
and
$$e_h v_j\otimes v_i e_l = \sum_{g\in G} g^{-1}(v_j)\otimes  g^{-1}(v_i) a_g g e_l.$$
So we see that $e_h v_i\otimes v_j e_l \neq 0$ if and only if $e_h
v_j\otimes v_i e_l \neq 0$ in $T_kV$. This proves the lemma.

\end{proof}

Now regard $\Lambda(G)$ as the quotient algebra $\Lambda_Q \simeq kQ_{G}/(\rho_G)$, the
relation $\rho_G$ comes from those of the exterior algebra. That is, for the
basis $v_1, \ldots, v_m$ of $V$, for all $s,t$ we have in $\wedge V$ the relation
$$v_s v_t + v_t v_s =0$$ for all $s, t$. Then we have $\rho_G$ consists of the
relations $e_jv_s v_t e_i+ e_jv_t v_se_i $ for all $i,j \in Q_0$ and all $s,t
$. We have that $v_{i_1} \ldots v_{i_r} = 0$ in $\wedge V$ if and only if there are
$1\le h< h' \le r$ such that $i_{h}= i_{h'}$.

\begin{lemma}\label{extNak} The Nakayama translation is extended uniquely to an
automorphism of the quiver.
\end{lemma}

\begin{proof} Since the socle of an indecomposable projective $\Lambda(G)$-module is simple,
we have that $\mathrm{dim}_k e_{\sigma i} \Lambda_m e_i =1$. Now for any arrow $\alpha: i
\to j$ of type $t$, $\alpha = e_{j}v_t e_i $. We have that in $\wedge V$
$v_{i_s}\ldots v_{i_1}$ =0 if there are $1 \le h<h'\le s$ such that $i_h =
i_{h'}$. This implies that $0 \neq \alpha_m \ldots \alpha_2\alpha \in e_{\sigma
i}\Lambda_m e_i$ if and only if for $h=2, \ldots, m$, $\alpha_h = e_{j_h} v_{t_h}
e_{i_h}\neq 0$, $i_1=i, j_1=j$ such that $i_{h+1}=j_h$, and $t_1,\ldots,t_m$
are pairwise different. This shows that $\beta\alpha_m \ldots \alpha_2 \neq 0$
implies $\beta = e_l v_{t_1} e_{\sigma i}$. On the other hand, since $0 \neq
\alpha_m \ldots \alpha_2\in \Lambda_{m-1} e_{i_2}$ is not in $\mathrm{soc}
\Lambda_{m-1} e_{i_2}$ and $0\neq V \alpha_m \ldots \alpha_2\in \mathrm{soc}
\Lambda_{m-1} e_{i_2}$, so such $\beta$ exists and we have that $l = \sigma i_2$.
Define $\sigma \alpha = \beta$, this is an extension of $\sigma$, and for any
vertex $i$ of $Q_G$, $\sigma $ induces an bijection from $i^+$ to $(\sigma i)^+
$. So $\sigma$ is extended to an automorphism of the quiver $Q_G$. \end{proof}

A relation in a quiver $Q$ is a subset of elements in $(kQ^2)$, the
ideal of the path algebra $kQ$ generated by paths of length $2$.
According to the definition of \cite{gre}, if $(Q, \rho)$ and $(Q',
\rho')$ are quivers with relations, a regular covering map $\pi: Q
\to Q'$ is called a morphism of quivers with relations if $\pi$
satisfies the following conditions:

\begin{enumerate}
\item $\rho = \{L(x)| L£ºQ' \to Q \mbox{ is a lifting and } x \in \rho' \}$

\item If $x\in \rho$ and $i',j' \in Q'_0$, there exists $i,j \in
Q_0$ such that $\pi(i)=i', \pi(j)=j'$ and $\bar{\pi}(e_j x e_i) =
e_{j'} \bar{\pi}(x )e_{i'}$, here $\bar{\pi}$ is the homomorphism
from $kQ$ to $kQ'$ induced by $\pi$, that is, if $x = \sum_t d_tp_t$
for $d_t \in k$ and $p_t$ paths in $Q$, then $\bar{\pi}(x ) = \sum_t
d_t\pi(p_t)$.
\end{enumerate}

Let $N$ be a normal subgroup of $G$ such that the conditions of
Theorem ~\ref{cover} are satisfied. Let $Q_G$ and $Q_N$ be the McKay
quivers of $G$ and $N$, respectively and let $\pi: Q_G \to Q_N$ be
the covering map. Let $\Lambda(G)$ and $\Lambda(N)$ be their basic algebras,
respectively, then $\Lambda(G) \simeq  k Q_G / (\rho_G)$ and $\Lambda(N)
\simeq k Q_N / (\rho_N)$. Now consider the relationship between
$\rho_N$  and $\rho_G$. We have the following theorem

\begin{thm}\label{coverqwr} Let $G$ and $N$ as in Theorem ~\ref{cover}. The covering
map $\pi :Q_G \to Q_N$ defined in the proof of Theorem ~\ref{cover} is a
morphism of quiver with relations. \end{thm}

\begin{proof} By Theorem ~\ref{cover}, we need only to prove the above conditions
(1) and (2).

Use the notations of the proof of Theorem ~\ref{cover}, if  $V =
\bigoplus_{i=1}^{n'} V_i$, where for $b_1+ \cdots +b_{t-1} < i \le
b_1+ \cdots +b_t$ we have $V_i \simeq S_t$ as irreducible
representation of $N$ and for $b_1+ \cdots +b_{t-1}+b_{t,1}+ \cdots
+b_{t,s-1} < i \le b_1+ \cdots +b_{t,1}+ \cdots +b_{t,s} $, $V_i
\simeq S_t\otimes T_{j_s}$ as irreducible representation of $G$. We
may assume that our basis $\{v_1, \ldots, v_m\}$ is a union of the
bases of $V_1, \ldots, V_{n'}$.

Now there is an arrow $\alpha$ from $i$ to $j$ in the quiver  $Q_N$
of $\Lambda(N)$ provided that we have a basic element $v_r$ in $V_{r'}$
such that  $S_j$ is a summand of $V_{r'} \otimes S_i$, in this case
$\alpha$ is of the form $e_{j}v_r e_i$. As we have done in the proof
of Theorem ~\ref{cover}, we may index the vertices (idempotents of $k G$)
of $Q_G$ as the quiver of $\Lambda(G)$ by the pair $\{(i,j)|i \in
Q_{N,0}, j\in \mathbb Z/|G/N|\mathbb Z \}$. If $i$ is lifted to
$(i,l)$ in $Q_G$, then the arrow $\alpha$ is lifted to an arrow
$\tilde{\alpha}: (i,l) \to (j,l+j')$, for some lift $(j,l+j')$ of
$j$ which is determined by the same element $v_r$ of $V_{r'}$
regarding as representation of $G$. Now both relations $\rho_N$ and
$\rho_G$ are quadratic and are induced by the relations
$\rho=\{v_tv_s+v_sv_t| 1\le t,s \le m\}$.  It follows from
Lemma~\ref{COMMARR} that both conditions (1) and (2) hold. \end{proof}

Obviously, the covering maps in Theorem ~\ref{cmpcover} can be regarded as
morphisms of quiver with relations. Denote  $\mathcal M_m(N)$. the
category of McKay quivers for the groups in $\mathcal G_m(L)$ with
the relations as above together with morphisms of quiver with
relations between them. It follows from Theorem 2.6 of \cite{gre},
there is an unique universal cover $(Q,\rho)$ for $\mathcal M_m(N)$,
it if an (locally finite) infinite quiver and hence is not the an
object in $\mathcal M_m(N)$.

In the setting of Theorem ~\ref{coverqwr}, regard $\Lambda(G)$ and $\Lambda(N)$ as
locally bounded categories with finitely many objects \cite{gr}, the
morphism $\pi$ of quiver with relation induces a covering functor
from $\Lambda(G)$ and $\Lambda(N)$.

\section{Finite Subgroups in $\mathrm{GL}(m,\mathbb C)$ and $\mathrm{SL}(m+1,\mathbb C)$}

Let $V$ be an $m+1-$dimensional vector space over $\mathbb C$ and
let $G$ be a finite subgroup of $\mathrm{GL}(m+1,\mathbb C) = \mathrm{GL}(V)$.
Let $\mathrm{det\,}$ denote the one-dimensional representation $\mathbb C$
defined by $g\cdot x =\mathrm{det\,}(g)x $, here $\mathrm{det\,}(g)$ denote the
determinant of $g$. Then $S \to S \otimes \mathrm{det\,}$ define a bijection
on the set of  irreducible representations, which induces an
automorphism on the McKay quiver. In the case of $m = 2$, this
coincide with the translation defined in \cite{g2}.

Let $V$ be an $m+1-$dimensional vector space over $\mathbb C$ and
let $V'$ be an $m-$dimensional subspace of $V$. Take a basis of $V'$
and extend it to a basis of $V$. In this way we embed $\mathrm{GL}(m,\mathbb
C) = \mathrm{GL}(V')$ into  $\mathrm{SL}(V) = \mathrm{SL}(m+1,\mathbb C)$ as follow. For
each $g \in \mathrm{GL}(V')$
$$f: g \to \left( \arr{cc}{g & 0\\ 0&\mathrm{det\,}(g)^{-1}}\right). $$
The follow Theorem tells us how their McKay quivers are related.

\begin{thm}\label{mpo} Let $G'$ be a finite subgroup of $ \mathrm{GL}(V')$ and let $G =
f(G') \subset \mathrm{SL}(V)$ be its image under the map defined above.
Then the McKay quiver of $G'$ is obtained from that of $G'$ by
adding an arrow from $\sigma i$ to $i$ for each vertex $i$ in
$Q_{G',0}$.
 \end{thm}
\begin{proof} Let $v_1,\ldots,v_{m}$ be a basis of $V'$ and $v_1, \ldots,
v_{m+1}$ be a basis of $V$. Consider the exterior algebras $\wedge V'$
and $\wedge V$. Then $\wedge V' \simeq \wedge V/(v_{m+1})$ naturally. Now consider
the skew group algebra $\wedge V*G$. Since the subspace $V'$ and subspace
$\langle v_{m+1}\rangle$  spanned by $v_{m+1}$ of $V$ are both
invariant under $G$. Consider now the ideal $(v_{m+1})$  generated
by $v_{m+1}$, so we have that
$$ \wedge V/(v_{m+1})*G = \wedge V/(v_{m+1}) *G   \simeq \wedge V'*G'
\subset  \wedge V/(v_{m+1})*G.$$ Since $G \simeq G'$, their McKay quiver
has the same number of vertices. In $\wedge V'*G'$, we have that the
image of $v_{m+1}$ is zero and the longest paths in it are formed by
$m$ arrows of different type going from each vertex to its Nakayama
translation, while in $\wedge V*G$  the longest paths in it are formed by
$m+1$ arrows of different type going from each vertex to itself by
\cite{gum}, these paths are formed by exactly the $m$ arrows of
different in $\wedge V'*G'$ adding one arrows of new type $v_{m+1}$. So
we see that for each vertex $i$, the new arrow in $Q_{G}$ is of type
$v_{m+1}$  going from the Nakayama translation $\sigma i$ of the
vertex back to $i$ itself. This proves our Theorem. \end{proof}

In this case, we say that the McKay quiver of $G$ is obtained from that of $G'$
by {\em replacing Nakayama translation with arrows.}

{\bf Remark.} Adding an new arrow in certain quiver to get a new one in a
similar way appears in the research on higher dimensional Auslander algebra and
the cluster algebras recently, see \cite{gls, iy}.

We have characterized the McKay quiver for finite abelian subgroup $G$ of
$\mathrm{SL}(m,\mathbb C)$ in \cite{g08}. Since the irreducible representations of an
abelian group are all one-dimensional, we can assume that all the elements in
the group are diagonal. So we can regard it as an image of a subgroup of
$\mathrm{GL}(m-1,\mathbb C)$, this shows the following Proposition.

\begin{prop}\label{abq} Let $G$ be an  abelian subgroup $G$ of $\mathrm{SL}(m,\mathbb C)$, then
there is a subgroup $G'$ of $\mathrm{GL}(m-1,\mathbb C)$, such that the McKay quiver of
$G$ is obtained from that of $G'$ by replacing Nakayama translation with
arrows. \end{prop}

\section{Examples}
In this section, we show  that certain interesting examples of McKay quivers
are obtained in these two ways. These quivers are very useful in mathematics,
it is interesting to know how  the construction can be used in the study of the
mathematical problems concerning these quivers.

\subsection{McKay quiver of the form double quiver of affine Dynkin diagram  $\tilde{A}_{n-1}$.}
The classical McKay quiver of type $\tilde{A}_{n-1}$ can be got by the theory
developed in this paper. Start with $V_1= \mathbb C$, let $G_1$ be a finite
subgroup of $\mathrm{SL}(1,\mathbb C) = \{1\}$, then $G_1$ is a trivial group, and we
know that its quiver $Q_{G_1}$ is just one loop. \setlength{\unitlength}{1mm}
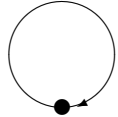
\begin{center}
\begin{picture}(20,20)
\put(10,7){\circle{20}} \put(10,0){\circle*{2}} \put(13,0.5){\vector(-2,-1){1}}
\end{picture}
\end{center}
 Now we make extension
of $G_1$ with a cyclic group $H$ of order $n+1$, get finite subgroup $G_2
\simeq H$ of $\mathrm{GL}(1,\mathbb C) = \mathbb C^*$, the McKay quiver $Q_{G_2}$ of
$G_2$ is a regular covering of $Q_{G_1}$, in fact it is a cyclic quiver of $n$
vertices.
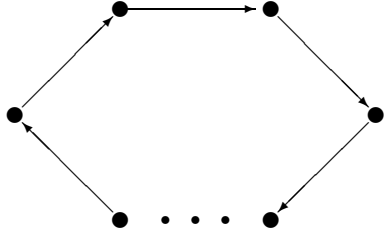
\begin{center}
\begin{picture}(50,42)
\put(14,2){\circle*{2}} \put(34,2){\circle*{2}} \put(0,16){\circle*{2}}
\put(48,16){\circle*{2}} \put(14,30){\circle*{2}} \put(34,30){\circle*{2}}
\put(20,2){\circle*{1}} \put(24,2){\circle*{1}} \put(28,2){\circle*{1}}
\put(15,30){\vector(1,0){17}} \put(35,29){\vector(1,-1){12}}
\put(47,15){\vector(-1,-1){12}} \put(13,3){\vector(-1,1){12}}
 \put(1,17){\vector(1,1){12}}
\end{picture}
\end{center} Now construct the skew group algebra $\wedge V*G$ with $V=V_1$, its basic
algebra $\Lambda(G_1)$ is a selfinjective algebra with vanishing radical square, so
the Nakayama translation for vertex $i$ is just the tail of the arrow starting
at $i$. Embed $V_1$ in $V_2= \mathbb C^2$ in the way as above, we get the
subgroup $G_3$ of $\mathrm{SL}(2,\mathbb C)$ isomorphic to $G_2$, whose McKay quiver
is exactly the double quiver of the affine Dynkin diagram of type
$\tilde{A}_{n-1}$.
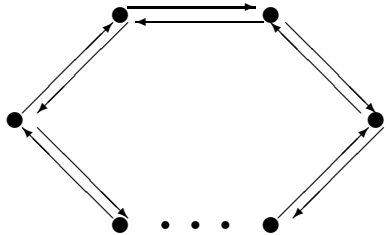
\begin{center}
\begin{picture}(50,42)
\put(14,2){\circle*{2}} \put(34,2){\circle*{2}} \put(0,16){\circle*{2}}
\put(48,16){\circle*{2}} \put(14,30){\circle*{2}} \put(34,30){\circle*{2}}
\put(20,2){\circle*{1}} \put(24,2){\circle*{1}} \put(28,2){\circle*{1}}
 \put(15,31){\vector(1,0){17}} \put(36,29){\vector(1,-1){12}}
 \put(49,15){\vector(-1,-1){12}} \put(13,3){\vector(-1,1){12}}
 \put(1,17){\vector(1,1){12}}
 \put(33,29){\vector(-1,0){17}} \put(46,17){\vector(-1,1){12}}
 \put(35,3){\vector(1,1){12}} \put(3,15){\vector(1,-1){12}}
 \put(15,29){\vector(-1,-1){12}}
\end{picture}
\end{center}

Since all the finite subgroup of $ \mathbb C^*$ are abelian, we see that  the
double quiver of  affine Dynkin diagram of type $\tilde{A}_{n-1}$ are the only
ones obtained from that of subgroups of $\mathrm{GL}(1,\mathbb C) = \mathbb C^*$ by
replacing Nakayama translation with arrows. So we get the following
proposition.

\begin{prop}\label{affdk} The double quiver of an affine Dynkin diagram is  obtained from
the McKay quiver of subgroups of $ \mathbb C^*$ by replacing Nakayama
translation with arrows if and only if it is of type $\tilde{A}_n$.\end{prop}

\subsection{From double quiver of type $\tilde{A}_1$ to   $3$-McKay quivers with $4$ vertices}

Start with the order $2$ cyclic subgroup generated by the image of
$\matr{cc}{-1&0\\0&-1}$ in $\mathrm{SL}(2,\mathbb C)$, its McKay quiver is the  double
quiver of type $\tilde{A}_1$

\medskip
\setlength{\unitlength}{1mm}

Extend with order $2$ subgroups in $\mathrm{GL}(2,\mathbb C)/\mathrm{SL}(2,\mathbb
C)$ generated by the image of $\matr{cc}{i&0\\0&i}$ and
$\matr{cc}{1&0\\0&-1}$, respectively, we get two order $4$ subgroups
in $\mathrm{GL}(2,\mathbb C)$ with the following McKay quivers,
respectively.

\medskip
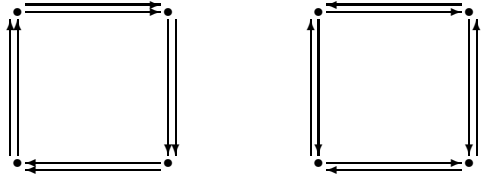
\setlength{\unitlength}{1mm}

\begin{center}
\begin{picture}(80,25)
\put(5,2){\circle*{1}} \put(5,22){\circle*{1}}
\put(25,2){\circle*{1}} \put(25,22){\circle*{1}}
\put(45,2){\circle*{1}} \put(45,22){\circle*{1}}
\put(65,2){\circle*{1}} \put(65,22){\circle*{1}}

\put(4,3){\vector(0,1){18}} \put(5,3){\vector(0,1){18}}
\put(6,22){\vector(1,0){18}} \put(6,23){\vector(1,0){18}}
\put(25,21){\vector(0,-1){18}} \put(26,21){\vector(0,-1){18}}
\put(24,1){\vector(-1,0){18}} \put(24,2){\vector(-1,0){18}}

\put(44,3){\vector(0,1){18}} \put(45,21){\vector(0,-1){18}}
\put(46,22){\vector(1,0){18}} \put(64,23){\vector(-1,0){18}}
\put(65,21){\vector(0,-1){18}} \put(66,3){\vector(0,1){18}}
\put(64,1){\vector(-1,0){18}} \put(46,2){\vector(1,0){18}}

\end{picture}
\end{center}

Note the second one is exactly the double $\tilde{A}_3$ quiver, which is the
same as the McKay quiver of order $4$ cyclic subgroup of $\mathrm{SL}(2,\mathbb C)$.
The difference between the two lies in  their Nakayama translations. While
Nakayama translation for a subgroup of $\mathrm{SL}(2,\mathbb C)$ is identity, it
sends each vertex to the opposite one in our second McKay quiver. This also
shows that the Nakayama translation should be essential ingredient when
consider McKay quiver for finite subgroups of $\mathrm{GL}(m,\mathbb C)$.

Now embedding these subgroups in $\mathrm{SL}(3,\mathbb C)$, using the map $f$ defined
above, we add an arrow for each vertex from its Nakayama translation back to
itself and get respectively the following McKay quivers for them. They are just
the only 3-McKay quivers with $4$ vertex for  finite subgroups of
$\mathrm{GL}(m,\mathbb C)$ given in \cite{g08}.
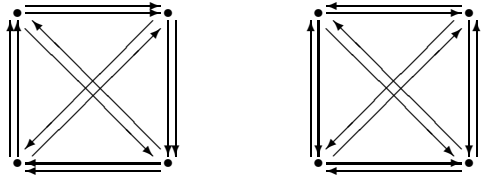
\begin{center}
\begin{picture}(10,25)
\put(5,2){\circle*{1}} \put(5,22){\circle*{1}}

\put(3,3){\vector(0,1){18}} \put(4,3){\vector(0,1){18}}
\put(6,21){\vector(0,-1){18}} \put(7,21){\vector(0,-1){18}}

\end{picture}
\end{center}
\medskip

\setlength{\unitlength}{1mm}

\begin{center}
\begin{picture}(80,25)
\put(5,2){\circle*{1}} \put(5,22){\circle*{1}}
\put(25,2){\circle*{1}} \put(25,22){\circle*{1}}
\put(45,2){\circle*{1}} \put(45,22){\circle*{1}}
\put(65,2){\circle*{1}} \put(65,22){\circle*{1}}

\put(4,3){\vector(0,1){18}} \put(5,3){\vector(0,1){18}}
\put(6,22){\vector(1,0){18}} \put(6,23){\vector(1,0){18}}
\put(25,21){\vector(0,-1){18}} \put(26,21){\vector(0,-1){18}}
\put(24,1){\vector(-1,0){18}} \put(24,2){\vector(-1,0){18}}

\put(7,3){\vector(1,1){17}} \put(23,21){\vector(-1,-1){17}}
\put(6,20){\vector(1,-1){17}} \put(24,4){\vector(-1,1){17}}

\put(44,3){\vector(0,1){18}} \put(45,21){\vector(0,-1){18}}
\put(46,22){\vector(1,0){18}} \put(64,23){\vector(-1,0){18}}
\put(65,21){\vector(0,-1){18}} \put(66,3){\vector(0,1){18}}
\put(64,1){\vector(-1,0){18}} \put(46,2){\vector(1,0){18}}

\put(47,3){\vector(1,1){17}} \put(63,21){\vector(-1,-1){17}}
\put(46,20){\vector(1,-1){17}} \put(64,4){\vector(-1,1){17}}

\end{picture}
\end{center}

\subsection{An example of a quiver in the study of D-branes.}
In their work \cite{gj}, Govindarajan and Jayaraman use two McKay quivers and
Beilinson quivers to describe the D-branes at the orbifold point. Here we show
how one of their quivers are constructed from some lower dimensional ones,
using the theory of this paper, the other one also has a similar construction.
The McKay quiver for $\mathbb P^{1,1,1,1,2}$ in \cite{gj} can be obtain as
follow. Start with the trivial subgroup of $\mathrm{SL}(4,\mathbb C)$, one get McKay
quiver with one vertex and four loops,
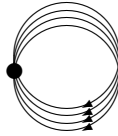
\begin{center}
\begin{picture}(20,20) \put(3,5){\circle*{2}}
\put(10,7){\circle{20}}\put(10,6){\circle{18}}\put(10,5){\circle{16}}\put(10,4){\circle{14}}
\put(13,0.5){\vector(-2,-1){1}} \put(13,-1.5){\vector(-2,-1){1}}
\put(13,-0.5){\vector(-2,-1){1}} \put(13,-2.5){\vector(-2,-1){1}}
\end{picture}
\end{center}
extending it with the
subgroup of $\mathrm{GL}(4,\mathbb C)$ generated by the scalar matrix with scalar
$\xi_6$, a $6$th primitive root of the unit. One gets the McKay quiver with
oriented cyclic quiver with $6$ vertices, indexed by $\mathbb Z_6$, with $4$
arrows from $i$ to $i+1$ for each $i \in\mathbb Z_6$.
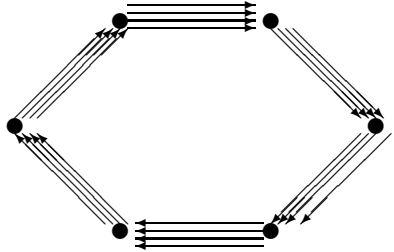
\begin{center}
\begin{picture}(50,35)
\put(14,2){\circle*{2}} \put(34,2){\circle*{2}} \put(0,16){\circle*{2}}
\put(48,16){\circle*{2}} \put(14,30){\circle*{2}} \put(34,30){\circle*{2}}
 \put(15,32){\vector(1,0){17}}  \put(15,31){\vector(1,0){17}}  \put(15,30){\vector(1,0){17}}  \put(15,29){\vector(1,0){17}}
 \put(34,29){\vector(1,-1){12}} \put(35,29){\vector(1,-1){12}} \put(36,29){\vector(1,-1){12}} \put(37,29){\vector(1,-1){12}}
 \put(50,15){\vector(-1,-1){12}} \put(48,15){\vector(-1,-1){12}} \put(47,15){\vector(-1,-1){12}} \put(46,15){\vector(-1,-1){12}}
 \put(12,3){\vector(-1,1){12}} \put(13,3){\vector(-1,1){12}} \put(14,3){\vector(-1,1){12}} \put(15,3){\vector(-1,1){12}}
 \put(0,17){\vector(1,1){12}}  \put(1,17){\vector(1,1){12}}\put(2,17){\vector(1,1){12}} \put(3,17){\vector(1,1){12}}
 \put(33,0){\vector(-1,0){17}}  \put(33,1){\vector(-1,0){17}}  \put(33,2){\vector(-1,0){17}}  \put(33,3){\vector(-1,0){17}}
\end{picture}
\end{center}In this case, the
Nakayama translation sending $i$ to $i+4$, embedding in $\sl(4,\mathbb C)$, the
new McKay quiver adds for each $i$ an arrow from $i+4$ to $i$, this is exactly
the McKay quiver for $\mathbb P^{1,1,1,1,2}$ in \cite{gj}.
\begin{center}
\begin{picture}(50,35)
\put(14,2){\circle*{2}}  \put(0,16){\circle*{2}}  \put(14,30){\circle*{2}}
\put(30,30){\circle*{2}} \put(44,16){\circle*{2}} \put(30,2){\circle*{2}}
 \put(15,32){\vector(1,0){13}}  \put(15,31){\vector(1,0){13}}  \put(15,30){\vector(1,0){13}}  \put(15,29){\vector(1,0){13}}
 \put(30,29){\vector(1,-1){12}} \put(31,29){\vector(1,-1){12}} \put(32,29){\vector(1,-1){12}} \put(33,29){\vector(1,-1){12}}
 \put(46,15){\vector(-1,-1){12}} \put(43,15){\vector(-1,-1){12}} \put(44,15){\vector(-1,-1){12}} \put(45,15){\vector(-1,-1){12}}
 \put(12,3){\vector(-1,1){12}} \put(13,3){\vector(-1,1){12}} \put(14,3){\vector(-1,1){12}} \put(15,3){\vector(-1,1){12}}
 \put(0,17){\vector(1,1){12}}  \put(1,17){\vector(1,1){12}}\put(2,17){\vector(1,1){12}} \put(3,17){\vector(1,1){12}}
 \put(29,0){\vector(-1,0){13}}  \put(29,1){\vector(-1,0){13}}  \put(29,2){\vector(-1,0){13}}  \put(29,3){\vector(-1,0){13}}

 \put(16,28){\vector(2,-1){25}}
 \put(30,27){\vector(0,-1){22}}
 \put(41,16){\vector(-2,-1){25}}
 \put(14,5){\vector(0,1){22}}
 \put(3,16){\vector(2,1){25}}
 \put(29,4){\vector(-2,1){25}}
\end{picture}
\end{center}

{}
\end{document}